\documentclass[12pt]{amsart}
\usepackage{amsmath,amssymb,enumerate}
\usepackage[usenames, dvipsnames]{color}
\newtheorem{theorem}{Theorem}[section]

\newtheorem{lemma}[theorem]{Lemma}

\newtheorem{corollary}[theorem]{Corollary}
\newtheorem{conjecture}[theorem]{Conjecture}
\theoremstyle{definition}

\newcommand{\cL}{\mathcal{L}}

\newcommand{\cU}{\mathcal{U}}

\DeclareMathOperator{\PG}{PG}
\DeclareMathOperator{\GF}{GF}

\newcommand{\del}{\setminus \!}

\allowdisplaybreaks[1]
\author{Adam Brown and Peter Nelson }
\address{Department of Combinatorics and Optimization,
University of Waterloo, Waterloo, Canada}
\title[On the number of hyperplanes in a matroid]{Matroids with no $U_{2,n}$-minor and many hyperplanes}
\thanks{This research was supported by a Discovery Grant from the Natural Sciences and Engineering Research Council of Canada}
\begin{document}
\begin{abstract}
We construct, for every $r \ge 3$ and every prime power $q > 10$, a rank-$r$ matroid with no $U_{2,q+2}$-minor, having more hyperplanes than the rank-$r$ projective geometry over $\GF(q)$.
\end{abstract}
\maketitle

\section{Introduction}

This note considers the following special case of a conjecture due to Bonin; see Oxley [\ref{oxley}, p. 582].
\begin{conjecture}\label{bonin}
If $q$ is a prime power and $M$ is a rank-$r$ matroid with no $U_{2,q+2}$-minor,
then $M$ has at most $\tfrac{q^r-1}{q-1}$ hyperplanes.
\end{conjecture}
The conjectured bound is attained by the projective geometry $\PG(r-1,q)$, and is also equal to the number of points in $\PG(r-1,q)$; an analogous upper bound on the number of points in a matroid with no $U_{2,q+2}$-minor was proved by Kung [\ref{kungextremal}], and the conjecture seems natural given the symmetry between points and hyperplanes in a projective geometry. The conjecture was also supported by a result of the second author [\ref{nelson}] stating that, for a fixed $k$ and large $r$, the number of rank-$k$ flats in a rank-$r$ matroid with no $U_{2,q+2}$-minor does not exceed the number of rank-$k$ flats in a projective geometry. 

However, Conjecture~\ref{bonin} fails; Geelen and Nelson [\ref{gn}] gave counterexamples for $r=3$ and $q \ge 7$. As observed in [\ref{gn}], it still seemed plausible that those rank-$3$ counterexamples were the only ones: as in the problem of classifying projective planes, sporadic behaviour in rank $3$ that disappears for larger rank can easily occur. 
We show using a variant of the construction in [\ref{gn}] that in fact, the conjecture fails much more dramatically.
\begin{theorem}\label{main}
For all integers $r \geq 4$ and $\ell \ge 10$ there exists a rank-$r$ matroid $M$ with no $U_{2,\ell+2}$ minor and more than $\frac{\ell^r-1}{\ell-1}$ hyperplanes.
\end{theorem}

In fact, for large $r$ and $\ell$ our counterexamples contain at least $(c\ell)^{3r/2}$ hyperplanes for some absolute constant $c \approx 2^{-7}$. In light of this, is not obvious what the correct upper bound should be; while it seems difficult to asymptotically improve on the construction we use, our counterexamples are not even $3$-connected, so quite possibly richer matroids with more hyperplanes exist. Our order-$(cq)^{3/2r}$ construction still has many fewer hyperplanes than the upper bound of $q^{r(r-1)}$ given in [\ref{g}]. However, we cautiously conjecture that projective geometries give the correct upper bound in the case of very high rank and connectivity; a matroid is \emph{round} if its ground set is not the union of two hyperplanes, or equivalently if it is vertically $k$-connected for all $k$. 

\begin{conjecture}
	Let $\ell \ge 2$ be an integer. If $M$ is a round matroid with sufficiently large rank and with no $U_{2,\ell+2}$-minor, then $M$ has at most $\tfrac{\ell^{r(M)-1}}{\ell-1}$ hyperplanes.
\end{conjecture}

\section{Rank Three}\label{smallsection}

We follow the notation of Oxley [\ref{oxley}]. If $M_1,M_2$ are matroids with $E(M_1) \cap E(M_2)$ equal to $\{e\}$ for some nonloop $e$ in both matroids, then the \emph{parallel connection} of $M_1$ and $M_2$, which we denote $M_1 \oplus_e M_2$, is the unique matroid $M$ on ground set $E(M_1) \cup E(M_2)$ for which $M|E(M_i) = M_i$ for each $i$, and $M \del e$ is the $2$-sum of $M_1$ and $M_2$. We write $\cU(\ell)$ for the class of matroids with no $U_{2,\ell+2}$-minor. If $e \in E(M)$ then $W_2(M)$ denotes the number of lines of $M$, and $W_2^e(M)$ denotes the number of lines of $M$ \emph{not} containing $e$.

\begin{lemma}\label{twosums}
	If $\ell \ge 1$, then $\cU(\ell)$ is closed under parallel connections. 
\end{lemma}
\begin{proof}
	Let $M,N \in \cU(\ell)$ with $E(M) \cap E(N) = \{e\}$ and suppose for a contradiction that $M \oplus_e N$ has a minor $L \cong U_{2,\ell+2}$. Note that $L$ has the form $M' \oplus N'$ or $M' \oplus_e N'$ for some minors $M'$ and $N'$ of $M$ and $N$ respectively. Since $L$ is $3$-connected, either $M' = \varnothing$ or $N' = \varnothing$, so $L$ is a minor of $M$ or $N$, a contradiction. 
\end{proof}

We now construct counterexamples to Conjecture~\ref{bonin}. 
This first construction appears in [\ref{gn}] attributed to Blokhuis. We include the proof, which is repeated essentially verbatim, for completeness. 

\begin{lemma}\label{construction}
	Let $q \ge 3$ be a prime power and $t$ be an integer with $3 \le t \le q$. There is a rank-$3$ matroid $M(q,t)$ with no $U_{2,q+t}$-minor such that $W_2(M(q,t)) = q^2 + (q+1)\tbinom{t}{2}$. 
\end{lemma}
\begin{proof}
	Let $N \cong \PG(2,q)$. Let $e \in E(N)$ and let $L_1,L_2,L_3$ be distinct lines of $N$ not containing $e$ and so that $L_1 \cap L_2 \cap L_3$ is empty. Note that every line of $M$ other than $L_1, L_2$ and $L_3$ intersects $L_1 \cup L_2 \cup L_3$ in at least two and at most three elements. 
	
	Let $\cL$ be the set of lines of $N$ and $\cL_e$ be the set of lines of $N$ containing $e$. For each $L \in \cL_e$, let $T(L)$ be a $t$-element subset of $L - \{e\}$ containing $L \cap (L_1 \cup L_2 \cup L_3)$. Observe that the $T(L)$ are pairwise disjoint. Let $X = \cup_{L \in \cL_e} T(L)$, noting that $L_1 \cup L_2 \cup L_3 \subseteq X$ and so each line in $\cL$ intersects $X$ in at least two elements. Note that a simple rank-$3$ matroid is completely determined by its set of lines, as every nonspanning circuit is contained in such a line. Let $M(q,t)$ be the simple rank-$3$ matroid with ground set $X$ whose set of lines is $\cL_1 \cup \cL_2$, where $\cL_1 = \{L \cap X: L \in \cL - \cL_e\}$, and $\cL_2$ is the collection of two-element subsets of the sets $T(L)$ for $L \in \cL_e$. Note that $\cL_1$ and $\cL_2$ are disjoint. Every $f \in X$ lies in $q$ lines in $\cL_1$ and in $(t-1)$ lines in $\cL_2$, so $M(q,t)$ has no $U_{2,q+t}$-minor. Moreover, we have $\cL_1 = |\cL - \cL_e| = q^2$ and $\cL_2 = |\cL_e|\binom{t}{2} = (q+1)\binom{t}{2}$. This gives the lemma. 
	\end{proof}

The following lemma slightly strengthens one in [\ref{gn}].

\begin{lemma}\label{counterexamples}
	If $\ell$ is an integer with $\ell \ge 10$, then there exists $M \in \cU(\ell)$ such that $r(M) = 3$ and $W_2(M) > \ell^2 + \tfrac{7}{3}\ell + 4$. 
\end{lemma}
\begin{proof}
	 If $\ell \ge 127$, let $q$ be a power of $2$ such that $\tfrac{1}{4}(\ell+2) < q \le \tfrac{1}{2}(\ell+2)$.  We have $2q \le \ell+2$ so $M(q,q) \in \cU(\ell)$, and 
	 \[W_2(M(q,q)) = q^2 + \tbinom{q}{2}(q+1) >\tfrac{1}{2}q^3 >  \tfrac{1}{128} (\ell+2)^3 > (\ell+2)^2.\] 

	If $10 \le \ell < 127$, then it is easy to check that there is some prime power $q \in \{7,9,13,19,32,59,113\}$ such that $\tfrac{1}{2}(\ell+2) \le q \le \ell-3$. Note that $4 < \ell+2-q \le q$. Define real quadratic polynomials $f_q(x)$ by $f_q(x) = q^2 + (q+1)\tbinom{x+2-q}{2}$ and $g(x) = x^2+\tfrac{7}{3}x + 4$. The function $h(x) = f_q(x)-g(x)$ has positive leading coefficient and $h(q+1) < 0$, while $h(q+3) = \tfrac{5}{3}q	 -1 > 0$; thus $h(x) > 0$ for every integer $x \ge q+3$. Now the matroid $M = M(q,\ell+2-q)$ satisfies $M \in \cU(\ell)$ and $W_2(M) - (\ell^2 + \tfrac{7}{3}\ell+4) = h(\ell) > 0$. 
\end{proof}\

\section{Large Rank}\label{highsection}
We now give a construction that extends rank-$3$ counterexamples to counterexamples in arbitrary rank.

\begin{lemma}\label{mr}
	Let $\ell \ge 2$ be an integer. Let $N \in \cU(\ell)$ be a rank-$3$ matroid and $e \in E(N)$. Then for each integer $r \ge 3$ there is a rank-$r$ matroid $M_r \in \cU(\ell)$ such that
	\begin{itemize}
		\item if $r$ is odd, then $M_r$ has at least $(W_2^e(M))^{(r-1)/2}$ hyperplanes, 
		\item if $r$ is even, then $M_r$ has at least $\ell (W_2^e(M))^{(r-2)/2}$ hyperplanes. 
	\end{itemize}
\end{lemma}
\begin{proof}
	For each $r \ge 3$, let $k = \left\lfloor\tfrac{1}{2}(r-1)\right\rfloor$ and let $N_1,\dotsc,N_k,L$ be matroids, any two of whose ground sets have intersection $\{e\}$, such that each $N_i$ is isomorphic to $N$ under an isomorphism fixing $e$, while $L \cong U_{2,\ell+1}$. For each odd $r \ge 3$, let $M_r$ be the parallel connection of $N_1, \dotsc, N_k$, and for each even $r \ge 4$, let $M_r$ be the parallel connection of $M_{r-1}$ and $L$. By Lemma~\ref{twosums} we have $M_r \in \cU(\ell)$ for all $r$. Note that $r(M_r) = r$ for each $r$. 
	
	Let $L_1, \dotsc, L_k$ be lines of $N_1, \dotsc, N_k$ respectively that do not contain $e$. If $r$ is odd, then $\cup_{i=1}^k L_i$ is a hyperplane of $M_r$; it follows that $M_r$ has at least $(W_2^e(N))^k = (W_2^e(N))^{(r-1)/2}$ hyperplanes, as required. If $r$ is even, then for each $x \in L-\{e\}$ the set $\cup_{i=1}^k L_i \cup \{x\}$ is a hyperplane of $M_r$. Thus $M_r$ has at least $|L-\{e\}|(W_2^e(N))^k = \ell (W_2^e(N))^{(r-2)/2}$ hyperplanes. 
\end{proof}
Using this, we can restate and prove Theorem~\ref{main}.

\begin{theorem}
	If $r \ge 4$ and $\ell \ge 10$ are integers, then there is a rank-$r$ matroid $M \in \cU(\ell)$ having more than $\tfrac{\ell^r-1}{\ell-1}$ hyperplanes. 
\end{theorem}
\begin{proof}
	Let $\ell \ge 10$ and $r \ge 3$. Let $N \in \cU(\ell)$ be a matroid given by Lemma~\ref{counterexamples} for which $r(N) = 3$ while $W_2(N) > \ell^2+\tfrac{7}{3}\ell+4$. Let $e \in E(N)$; since $e$ is in at most $\ell+1$ lines we have \[W_2^e(N) > (\ell^2+\tfrac{7}{3}\ell+4) - (\ell+1) > (\ell+\tfrac{2}{3})^2.\] Let $M_r$ be the matroid given by Lemma~\ref{mr}. If $r$ is odd, then $M_r$ has at least $(W_2^e(N))^{(r-1)/2} > (\ell+\tfrac{2}{3})^{r-1}$ hyperplanes. If $r$ is even, then $M_r$ has at least $\ell(W_2^e(N))^{(r-2)/2} > \ell(\ell+\tfrac{2}{3})^{r-2}$ hyperplanes. An easy induction on $r$ verifies that $\min((\ell+\tfrac{2}{3})^{r-1},\ell(\ell+\tfrac{2}{3})^{r-2}) > \tfrac{\ell^r-1}{\ell-1}$ for all $r \ge 4$, and the result follows. 
\end{proof}

Finally, we show that for large $r$ and $\ell$, we can construct examples having dramatically more than $\tfrac{\ell^r-1}{\ell-1}$ hyperplanes. Using the fact that for all $\epsilon > 0$ and all large $\ell$, there is a prime between $(1-\epsilon)\ell$ and $\ell$, one could improve the constant to anything under $2^{-4}$ for large $r$. 

\begin{corollary}
	If $r \ge 3$ and $\ell \ge 10$ are integers, then there is a rank-$r$ matroid $M \in \cU(\ell)$ having at least $(2^{-7}\ell^3)^{(r-2)/2}$ hyperplanes. 
\end{corollary}
\begin{proof}
	Let $q \ge 5$ be a prime power so that $\tfrac{1}{4}(\ell+2) < q \le \tfrac{1}{2}(\ell+2)$. Let $N  = M(q,q)$ as defined in Lemma~\ref{construction}. We have $N \in \cU(\ell)$ and $W_2(N) = q^2 + (q+1)\tbinom{q}{2} \ge \tfrac{1}{2}q^3 + 4q > 2^{-7}\ell^3 + \ell+1$, where we use $q \ge 5$. Let $e \in E(N)$; since $e$ is in at most $\ell+1$ lines of $N$ we have $W_2^e(N) > 2^{-7}\ell^3$. The result follows from Lemma~\ref{mr}. 
\end{proof}

\section*{References}

\newcounter{refs}

\begin{list}{[\arabic{refs}]}
{\usecounter{refs}\setlength{\leftmargin}{10mm}\setlength{\itemsep}{0mm}}

\item\label{g}
J. Geelen,
Small cocircuits in matroids,
European J. Combin. 32 (2011), 795-801.

\item\label{gn}
J. Geelen, P. Nelson, 
The number of lines in a matroid with no $U_{2,n}$-minor,
European J. Combin. 50 (2015), 115-122.

\item\label{kungextremal}
J.P.S. Kung,
Extremal matroid theory, in: Graph Structure Theory (Seattle WA, 1991), 
Contemporary Mathematics 147 (1993), American Mathematical Society, Providence RI, ~21--61.

\item\label{nelson}
P. Nelson, 
The number of rank-$k$ flats in a matroid with no $U_{2,n}$-minor,
J. Combin. Theory. Ser. B 107 (2014), 140-147.

\item \label{oxley}
J. G. Oxley, 
Matroid Theory,
Oxford University Press, New York, 2011.

\end{list}
\end{document}